\documentclass[12pt]{amsart}

\usepackage{matyi}
\usetikzlibrary{decorations.pathreplacing}
\newcommand{\eps}{\ensuremath{\epsilon} }
\newcommand{\epso}{\ensuremath{\epsilon > 0\;}}
\newcommand{\ta}{\ensuremath{\tau}}
\newcommand{\N}{\ensuremath{\mathbb N} }

\newcommand{\ms}{{\medskip}}

\begin{document}
\title{Slowly Decaying Averages and Fat Towers}
 
\author{James T. Campbell}
\email{jcampbll@memphis.edu}
\author{M\'{a}t\'{e} Wierdl}
\email{mwierdl@memphis.edu}
\address{Department of Mathematical Sciences, Dunn Hall 373, University of Memphis, Memphis, TN 38152}

\maketitle

\section{Introduction and Statement of Results}
\label{sec:intro}

Let $(X,\Sigma,m,\tau)$ be an ergodic system, that is, $(X, \Sigma, m)$ is a probability space and $\tau: X \to X$ is an invertible ergodic $m$-preserving transformation. 
For a function $f:X\to\setR$, let $A_Nf$ denote the $N$th ergodic
average,
\begin{equation}
  \label{eq:7}
  A_Nf(x)=\frac{f(x)+\dots+\tau^ {N-1}f(x)}{N}. 
\end{equation}
Martin Barlow\footnote{Personal communication. M. Barlow: barlow@math.ubc.ca.} asked the following question, which arose from the work of a student (Zichun Ye) on interface models. \medskip


\begin{quest}\label{q:0}
   If $f(x) \ge 0$ is integrable, and 
\begin{equation}
\label{eq:mart}
N(x) = \min \{n: A_kf(x) \le 2 \int f \text{ for all }  k \ge n\},
\end{equation}
is it the case that $N(x)$ is also integrable? 
\end{quest}

The motivation for the question was that a positive answer would allow them to apply estimates obtained in \cite{ADS} (see especially Assumption 1.5) for so-called {\em transition densities} or {\em heat kernels}, to random walks in an ergodic random environment.

In this note we show that the answer to Question \ref{q:0} is no in general, even for bounded functions. In so doing we discover that every ergodic system has a special sort of Kakutani tower which we call a {\em fat tower} (Definition \ref{defn:1} and more generally, Definition \ref{defn:col}). 

We re-cast the question as follows. For non-negative $f(x)$ define the set $E_{f,N}$ by
\begin{equation}
  \label{eq:8}
  E_{f,N}=\cbrace*{x:A_Nf(x)> 2\int_Xf, \text{ and }
       A_nf(x)\le 2\int_Xf \text{ for }n>N}, 
\end{equation}
so that, with $N(x)$ as in \ref{eq:mart}, 
\[ \int_X N(x) \, dm \; = \sum_{N\ge 1}N\cdot  m\paren*{E_{f,N}}\, .\]
Thus an equivalent question is: If $f(x) \ge 0$ is integrable, is it the case that $\sum_{N\ge 1}N\cdot  m\paren*{E_{f,N}}$ is finite?
In the next section we show that in any system with a fat tower, we may find a bounded function (in fact, an indicator function) $f$ for which $\sum_{N\ge 1}N\cdot  m\paren*{E_{f,N}} = \infty$ (Theorem \ref{thm:1}). In the subsequent section we show how to take any given ergodic system and `inflate' it to produce an ergodic system with a fat tower. The final section contains a proof that in fact every aperiodic system on a non-atomic measure space has a fat tower  (Theorem \ref{thm:2}). 

\section{Construction of an indicator counterexample}
\label{sec:ind-count}

While a more general definition of fat tower is given in Section \ref{sec:inflate} (Definition \ref{defn:col}), in this section we work with the following slightly more restricted definition. A fat tower will be a special sort of
Kakutani return-time tower, so we begin  by recalling the Kakutani tower construction (\cite{Kak:43}). Let $B$ be any
set of nontrivial measure,  $0<m(B)<1$.  For each natural number $N$
define the {\em $N^{th}$ first return} set $B_N$ of $B$ by
  \begin{equation}
    \label{eq:9}
    B_N=\cbrace*{x:x\in B, \tau^Nx\in B\text{ and } \tau^nx\notin B\text{
        for } 0<n<N}
  \end{equation}
    For each $N$ consider the $N$th tower
  $T_N$ (which may be empty) defined by
  \begin{equation}
    \label{eq:10}
    T_N=\bigcup_{n<N}\tau^nB_N. 
  \end{equation}
 \begin{center}
    \begin{tikzpicture}
  [
  matyi arrow/.style = {->,>=stealth,thick},
  ]

  \coordinate (bBN) at (0,0);
  \coordinate (eBN) at (3,0);
  \coordinate (mBN) at (1.5,0);
  \draw[ultra thick,red] (0,0)--(3,0);
  \draw[ultra thick,red] (0,1)--(3,1);
  \draw[ultra thick,red] (0,2)--(3,2);
  \draw[dotted,red] (0,3)--(3,3);
  \draw[ultra thick,red] (0,4)--(3,4);
  \draw[dotted,red] (0,5)--(3,5);
  \draw[ultra thick,red] (0,6) --(3,6);
  \node[left] at (0,0) {$B_N$};
  \node[left] at (0,1) {$\tau B_N$};
  \node[left] at (0,2) {$\tau^2 B_N$};
  \node[left] at (0,4) {$\tau^{N/2-1} B_N$};
  \node[left] at (0,6) {$\tau^{N-1} B_N$};

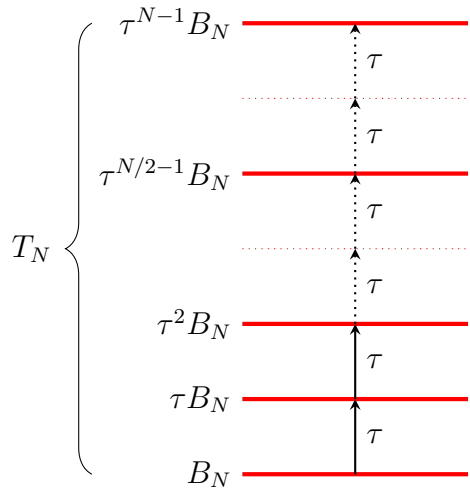
  \draw [decorate,decoration={brace,amplitude=10pt}]
  (-2,0) -- (-2,6)
  node [black,midway,xshift=-.8cm] { $T_N$};
  
  \draw[matyi arrow] (1.5,0)--(1.5,1) node[midway,right] {$\tau$};
  \draw[matyi arrow] (1.5,1)--(1.5,2) node[midway,right] {$\tau$};
  \draw[matyi arrow,dotted] (1.5,2)--(1.5,3) node[midway,right] {$\tau$};
  \draw[matyi arrow,dotted] (1.5,3)--(1.5,4) node[midway,right] {$\tau$};
  \draw[matyi arrow,dotted] (1.5,4)--(1.5,5) node[midway,right] {$\tau$};
  \draw[matyi arrow,dotted] (1.5,5)--(1.5,6) node[midway,right] {$\tau$};
\end{tikzpicture}

  \end{center}
  These towers are paiwise disjoint and their union $T=\bigcup_{N}T_N$
  covers $X$ up to a null set.  This $T$ is the {\em Kakutani tower over
  the base $B$}.
  \begin{defn}
    \label{defn:1}
    We say that a system $(X,\Sigma,m,\tau)$ has a \emph{fat tower} if
    there is a set $B$ with $0<m(B)<1$ so that the  Kakutani
    tower over the base $B$, $T=\bigcup_{N}T_N$, satisfies
    \begin{equation}
      \label{eq:15}
      \sum_NN\cdot m\paren*{T_N}=\infty. 
    \end{equation}
  \end{defn}
  Note that the condition in \cref{eq:15} is equivalent with
  \begin{equation}
    \label{eq:5}
    \sum_NN^2\cdot m\paren*{B_N}=\infty. 
  \end{equation}

\begin{thm}
  \label{thm:1}
  Suppose the system $(X,\Sigma,m)$ has the fat tower property.

  Then there's a set $A$ so that with $f=\mathbbm 1_A$ we have
  \begin{equation}
    \label{eq:6}
    \sum_{N\ge 1}N\cdot  m\paren*{E_{f,N}}=\infty. 
  \end{equation}
\end{thm}
\begin{proof}
  Let $B$ be the set guaranteed by the definition of the fat tower
  property, so we have
  \begin{equation}
    \label{eq:16}
     \sum_NN\cdot m\paren*{T_N}=\infty.
  \end{equation}
  Choose the integer $N_0$ large enough so that
  $m\paren*{ \bigcup_{N\ge N_0}T_N}<1/4$, and then the set $A$ is
  defined simply by
  \begin{equation}
    \label{eq:4}
    A=\bigcup_{N\ge N_0}T_N, 
  \end{equation}
  and hence we have
  \begin{equation}
    \label{eq:17}
    m(A)<1/4.
  \end{equation}
  From now on, unless we say otherwise, $N$ will always be assumed to
  satisfy $N\ge N_0$.  Let $f=\mathbbm 1_A$. We want to make a
  connection between the \emph{lower half} $L_N$ of the towers $T_N$
  and the sets $E_{f,N}$.  So we define
  \begin{align}
    \label{eq:12}
    L_N&=\bigcup_{n\in [0, N/2)} \tau^nB_N, \\
    L&=\bigcup_{N\ge N_0}L_N. 
  \end{align}
  \begin{center}
    \begin{tikzpicture}
  [
  matyi arrow/.style = {->,>=stealth,thick},
  ]

  \coordinate (bBN) at (0,0);
  \coordinate (eBN) at (3,0);
  \coordinate (mBN) at (1.5,0);
  \draw[ultra thick,red] (0,0)--(3,0);
  \draw[ultra thick,red] (0,1)--(3,1);
  \draw[ultra thick,red] (0,2)--(3,2);
  \draw[dotted,red] (0,3)--(3,3);
  \draw[ultra thick,red] (0,4)--(3,4);
  \draw[dotted,red] (0,5)--(3,5);
  \draw[ultra thick,red] (0,6) --(3,6);
  \node[left] at (0,0) {$B_N$};
  \node[left] at (0,1) {$\tau B_N$};
  \node[left] at (0,2) {$\tau^2 B_N$};
  \node[left] at (0,4) {$\tau^{N/2-1} B_N$};
  \node[left] at (0,6) {$\tau^{N-1} B_N$};

  \draw [decorate,decoration={brace,amplitude=10pt}]
  (-2,0) -- (-2,4)
  node [black,midway,xshift=-0.6cm] { $L_N$};
  \draw [decorate,decoration={brace,amplitude=10pt}]
  (-3,0) -- (-3,6)
  node [black,midway,xshift=-0.6cm] { $T_N$};
  
  \draw[matyi arrow] (1.5,0)--(1.5,1) node[midway,right] {$\tau$};
  \draw[matyi arrow] (1.5,1)--(1.5,2) node[midway,right] {$\tau$};
  \draw[matyi arrow,dotted] (1.5,2)--(1.5,3) node[midway,right] {$\tau$};
  \draw[matyi arrow,dotted] (1.5,3)--(1.5,4) node[midway,right] {$\tau$};
  \draw[matyi arrow,dotted] (1.5,4)--(1.5,5) node[midway,right] {$\tau$};
  \draw[matyi arrow,dotted] (1.5,5)--(1.5,6) node[midway,right] {$\tau$};
\end{tikzpicture}

  \end{center}
 As a consequence of \cref{eq:16} we have
  \begin{equation}
    \label{eq:14}
    \sum_{N\ge N_0}N\cdot m(L_N)=\infty. 
  \end{equation}
  We claim that this implies
  \begin{equation}
    \label{eq:1}
    \sum_{N\ge 1}N\cdot  m\paren*{E_{f,N}}=\infty. 
  \end{equation}
 Note that
  \begin{equation}
    \label{eq:13}
    A_Nf(x)\ge 1/2, \quad x\in L_N.  
  \end{equation}
  Since $\int_X f=m(A)<1/4$, we have, as a consequence of \cref{eq:13},
  that $L_N\subset \bigcup_{n\ge N} E_{f,n}$.  Let us set
  $C_{N,n}=L_N\cap E_{f,n}$ for $n\ge N$.  We have
  \begin{equation}
    \label{eq:2}
    N\cdot m(L_N)\le \sum_{n\ge N}n\cdot m(C_{N,n}).
  \end{equation}
  Since the $C_{N,n}$ are pairwise disjoint, summing \cref{eq:2} in
  $N$ over the range $N_0\le N\le K$ we get
  \begin{align}
    \label{eq:3}
    \sum_{N_0\le N\le K} N\cdot m(L_N)
    &\le \sum_{N_0\le N\le K} \sum_{n\ge N}n\cdot m(C_{N,n})\\
    &\le\sum_{n\ge N_0}n\cdot \sum_{N\le n} m(C_{N,n})\\
    \intertext{since $C_{N,n}\subset E_{f,n}$ for $N\le n$ and the
    $C_{N,n}$ are pairwise disjoint}
    &\le \sum_{n\ge N_0}n\cdot m\paren*{E_{f,n}}. 
  \end{align}
  Since $\lim_{K\to\infty}\sum_{N_0\le N\le K} N\cdot m(L_N)=\infty$ by
  \cref{eq:14}, we established \cref{eq:1}. 
\end{proof}

\section{Inflating a system to contain a fat tower}
\label{sec:inflate}

A fat tower need not be defined from a Kakutani return-time tower; the essential property is the growth of the measures of the columns described in (\ref{eq:15}). This is distilled into the following general definition:  

\begin{defn}\label{defn:col}
A {\em column of height k} (over a {\em base} $B$) for \ta \, consists of  pairwise disjoint sets $\{B, \ta B, \tau^2 B, \dots, \ta^{k-1}B\}$. The set $\ta^i B$ is known as the $i^{th}$ level of the column. 

A {\em tower} for \ta \; consists of a (finite or infinite) sequence of pairwise disjoint columns $\{C_n\}$ of heights $\{k_n\}$. 

 A {\em fat tower} for \ta \, is any tower for which 
\begin{equation} 
\label{eq:fatgen}
 \sum_{n = 1}^\infty k_n \cdot m(C_n) = \infty. 
\end{equation}
\end{defn}
 In this section we show how to take any ergodic system on a non-atomic probability space and `inflate' it to create a system which has a fat tower. 

For simplicity we consider a non-atomic, ergodic system $(X, \Sigma, m, \ta)$ where $(X, \Sigma, m)$ is the standard unit interval. We build the inflation $(Y, \mathcal B, \mu, T)$ by using $X$ as the base of a tower. Let $N_i = 2^i$ ($i = 1, 2, \dots$) and partition $X$  into disjoint intervals $B_i$ of $m$-measure $3/(2N_1), 3/(4N_2), \dots$, (so that $\sum_1^\infty m(B_i) = 1$). Above each $B_i$ place a column $C_i$ of height $N_i$, where each level in column $C_i$ is an interval of the same length as $B_i$.  $Y$ is the union of the columns. The sigma-algebra $\mathcal B$ is defined in the natural way. Define the transformation $T$ on each (non-top) level of each column as just moving up the column. The top level may be identified with the base in a natural way,  and $T$ sends the top level to where $\tau$  sent the base. This gives a measurable transformation on all of $Y$. 

 If we define $\mu$ first as agreeing with $m$ on the base and then extending it so that $T$ is $\mu$-preserving, we note that the measure of the column $C_i$ is $3/2^i$ ($i = 1, 2, \dots$), so the $\mu$-measure of $Y$ is $3$. Finally normalize $\mu$ so that $\mu(Y) = 1$. 

It is easily checked that $(Y, \mathcal B, \mu, T)$ is an ergodic system. Finally we point out that the columns $\{C_i\}$, which have heights $\{2^i\}$ and measures $\mu(C_i) = 2^{-i}$, clearly yield a fat tower. The construction of the indicator function given in Section \ref{sec:ind-count}, using the top-half of the columns in the fat tower, will again produce an example with $\int_Y N(y) d\mu(y) = \infty$. 

\section{Universal intrinsic fat tower construction}
\label{sec:fat}

Consider an arbitrary aperiodic system $(X, \Sigma, m, \tau)$, that is,  $(X, \Sigma, m)$ is a non-atomic probability space and $\ta: X \to X$ is an invertible, aperiodic $m$-preserving transformation. We prove
\begin{thm}
\label{thm:2}
Every aperiodic measure-preserving system on a non-atomic probability space possesses a fat tower. 
\end{thm} \ms

Thus, one could say that every aperiodic system is fat. 
 \begin{proof}[Proof of Theorem \ref{thm:2}:] We remind the reader that a {\em Kakutani-Rohlin} tower of height $k$ and error $\eps$ consists of a partition of $X$ into two sets, a column of height $k$ and an error set $E$, with $m(E) < \eps$. It is well-known that every aperiodic system possesses a Kakutani-Rohlin tower of height $k$ and error \eps for all $k \in \N$ and \epso (\cite{Kak:43}, \cite{Roh:52}.).

Fix a sequence of rapidly increasing natural numbers $1 = k_1 << k_2 << k_3 \dots$ whose minimal growth rate will be determined. \ms

We define a sequence of towers with the following three properties: 
\begin{enumerate}
\item[P1:] The $n^{th}$ tower partitions $X$. 
\item[P2:] The $n^{th}$ tower will have $n$ columns with heights $1, k_2, k_3, \dots, k_n$. 
\item[P3:] There is a constant $c > 0$ so that for all $j$ and $n$, the $j^{th}$ column of the $n^{th}$ tower has measure at least $c/k_j$.  
\end{enumerate}

 We then show that the limiting tower exists and has the same properties for all $n \in \N$, which is sufficient. \ms

First consider the following construction. For a natural number $k$, build (from all of $X$) a Kakutani-Rohlin tower of height $k+1$ and error $10^{-k}$. From this tower, extract a vertical slice of height $k$ whose total measure is $1/k$. That is,  take a measurable subset of the base with measure $1/k^2$ together with its images under the action of \ta \, for $k-1$ steps. We call this  {\em extracting a $(k, 1/k)$-column from $X$}. Clearly, it may be done for any $k \in \N$. \ms

Our sequence of towers satsifying P1 - P3 is defined as follows.  The first tower consists solely of $X$. \ms

Next, extract a $(k_2, 1/k_2)$-column $T_{(2, 1)}$ from $X$, and define our new tower as the pair of disjoint columns $\{(X \setminus T_{(2, 1)}), T_{(2, 1)}\} = \{X_2, T_{(2, 1)}\}$.  Note that $m(X_2) = 1 - 1/k_2$. Hence this second tower satisfies each of P1 - P3 above, with $c = 1 - 1/k_2$.  \ms

Now extract a $(k_3, 1/k_3$) column $T_{(3, 1)}$ from $X$. Some of this column may have come from $T_{(2,1)}$ and some of it from $X_2$. We therefore modify these columns as follows.

\begin{enumerate}
\item Delete from $X_2$ any portion of $T_{(3, 1)}$ which came from $X_2$. Call this new set $\tilde{X_2}$. The set deleted has measure at most $1/k_3$, so that $m(\tilde{X_2}) \ge 1 - 1/k_2 - 1/k_3$. 
\item Let $L_i$ denote the $i^{th}$ level of $T_{(2,1)}$. Define $C_0 = L_0 \cap T_{(3,1)}$ and for $i \ge 1$ define

\[C_i =   \{x \in L_i: x \in T_{(3,1)}\} \bigcap   \left[ \bigcup_{j = 1}^i \{x \in L_i: \tau^{-j}x \notin T_{(3,1)}\}\right] .\]
In other words, $C_i$ consists of the points in level $L_i$ which are in $T_{(3,1)}$, but none of whose pre-images in $T_{(2,1)}$ are in $T_{(3,1)}$. 

Now let 

\[ C = \bigcup_{i = 0}^{k_2-1} \;  \bigcup_{j = -i}^{k_2 - i - 1} \tau^j (C_i) .\]

Thus $C$ consists of all the forward and backward images of $C_i$ in the column $T_{(2,1)}$. Remove all of $C$ from $T_{(2, 1)}$, call what remains $T_{(2,2)}$.  \ms

Typically not all of $C$ will actually be in $T_{(3,1)}$ (that is, we have removed more from $T_{(2,1)}$ than just those pieces that were in $T_{(3,1)}$.) Each $C_i$, and any of the forward images of  $C_i$ which actually stay in the column $T_{(3,1)}$, will be in $T_{(3,1)}$ of course, but we cannot guarantee any more than that. Moreover we don't even know how many forward images of a given $C_i$ lie in $T_{(3,1)}$. It is possible for example that $C_0$ came from a level near the top of $T_{(3,1)}$ and therefore only a few forward images stay in $T_{(3,1)}$.  The reason we removed all of $C$ is to ensure that $T_{(2,2)}$ is a column (which the reader may check), thereby preserving the tower property.

Note that $T_{(2,2)}$ has measure at least $1/k_2 - k_2\cdot 1/k_3$. \ms

\item Finally,  {\em place back into $\tilde{X_2}$} any of $C$ which is not in $T_{(3,1)}$. Call this new set $X_3$.  
\end{enumerate}
Summarizing: We first extracted a $(k_3, 1/k_3)$ column $T_{(3,1)}$from $X$. Then we modified $X_2$ by removing any portion of $T_{(3,1)}$ which appears there, creating $\tilde{X_2}$.  Then we modified $T_{(2,1)}$ by removing any portion of $T_{(3,1)}$ which appears there, plus all its possible forward and backward images under \ta \, which remain in $T_{(2,1)}$. Then we created $X_3$ from  $\tilde{X_2}$ by placing back the portions removed from $T_{(2,1)}$ which were not actually in $T_{(3,1)}$. \ms

We observe that all of the employed set operations are finite combinations of intersection, union, and complementation applied to measurable sets, so that all sets under consideration are measurable. \medskip

At this stage we have a tower $\{X_3, T_{(2,2)}, T_{(3,1)}\}$ satisfying P1 - P3 above, although the constant $c$ in P3 has been reduced from the constant that went with the tower $\{X_2, T_{(2,1)}\}$. \ms

Inductively continue this process, so that at the $n^{th}$ stage we have: 
\begin{enumerate}
\item A tower consisting of columns $\{X_n, T_{(2, n-1)}, T_{(3, n-2)}, \dots, T_{(n, 1)}\}$.
\item Column $T_{(n,1)}$ was obtained by extracting a $(k_n, 1/k_n)$-column from $X$. 
\item The columns $T_{(j, n-j+1)}$ were obtained from the columns $T_{(j, n -j)}$ by removing a subset of measure at most $k_j \cdot 1/k_n$, $2 \le j \le n-1$. 
\item The set $X_n$ is obtained from $X_{n-1}$ by first removing a set of measure at most $1/k_n$ and then possibly adding another set.
\item The tower satisfies P1 \& P2.
\end{enumerate}
We need to ensure that P3 is satisfied for all the columns in this tower, for each $n$. But this is easily accomplished by specifying the growth of the $k_j$'s.\ms

The first column of the $n^{th}$ tower is $X_n$. It is obtained from $X_{n-1}$ by first removing a set of measure $1/k_n$, and then possibly adding some other sets back in. In any case, we see that its measure satisfies
\[ m(X_n) \ge 1 - \sum_{j = 2}^n 1/k_j\, .\]
Thus we first require that the $k_j$'s sum to less than 1/8, say.\ms

Now let's re-examine the construction of the other columns. If we fix $n$ and consider our tower at the $(n-1)^{st}$ stage, $\{X_{n-1}, T_{(2, n - 2)},  T_{(3,n-3)} \dots , T_{(n-1), 1)}\}$, we see that for each $j$, $2 \le j \le n-1$, we construct $T_{(j, n-j+1)}$ from $T_{(j, n-j)}$ by removing a set of measure at most $k_j\cdot 1/k_n$. In particular, the sequence of columns in the $j^{th}$ position are nested: 
\begin{equation}
  \label{eq:nest}
 T_{(j,1)} \supset T_{(j,2)} \supset \dots \supset T_{(j, n-j+1)} \, .  
\end{equation}

Thus, if we require 
\[ \forall j, \; \sum_{t = 1}^\infty k_j \cdot \frac{1}{k_{j+t}} \; < \; 1/8 \, , \]
then we will have, for all $n$ and all $2 \le j \le n$, 
\[m(T_{(j, n-j+1)} \ge 7/8 \cdot m(T_{(j, 1)}) = (7/8) \cdot (1/k_j)\, .\]
Thus P3 is satisfied with $c = 7/8$. \ms 

For the sake of definiteness, we set $k_j = 8^{2^j}$.

Now we discuss the limiting tower. First we consider what is happening for the $j^{th}$ columns with $j \ge 2$. Because of the nested property (\ref{eq:nest}), the limiting column $T_j$ defined by  
\[ T_j = \bigcap_{n=j}^\infty T_{(j, n-(j-1))}\]
is measurable, has height $k_j$, and measure at least $(7/8)\cdot (1/k_j)$. \ms 

Now we consider the sequence of first columns, $X = X_1, X_2, X_3, \dots$. These are not nested, necessarily. $X_n$ is constructed by removing a piece from $X_{n-1}$, and then adding some pieces back in. We claim that the set of points for which this occurs infinitely often is a nullset. The measure of the piece taken out of $X_{n-1}$ has measure at most $1/k_n$ while the measure of the pieces put back into $X_{n-1}$ is at most $\sum_{j = 2}^{n-1} k_j/k_n$. Thus the total measure of the points which are removed or replaced is 
\[ \sum_{n = 2}^\infty \left[ \frac{1}{k_n} \left(1 + \sum_{j = 2}^\infty \frac{k_j}{k_n} \right)\right]\, .\]
The reader may check that when $k_j = 8^{2^j}$, this sum is finite. Thus by the Borel-Cantelli Lemma, the set of points which move in or out infinitely often is a nullset. This means that almost all of $X \setminus \left(\bigcup_{j = 2}^\infty T_j\right)$ remains in each $X_n$ for sufficiently large $n$. We take this set as our limiting $X_{\infty}$. The tower $\{X_\infty, T_2, T_3, \dots\}$ now satisfies P1-P3, and is fat, concluding the proof of Theorem \ref{thm:2}.  \end{proof}

\bibliographystyle{amsalpha}
\bibliography{tower}

\end{document}